\documentclass[a4paper,12pt,english]{article}
\usepackage[utf8x]{inputenc}
\usepackage{amsmath, amsthm, amssymb, graphicx, makeidx, mathrsfs, enumerate}
\usepackage{setspace}
\usepackage{hyperref}
\usepackage[labelformat=empty]{caption}
\usepackage[T1]{fontenc} 
\usepackage{siunitx} 
\usepackage[multiple]{footmisc}
\usepackage[T1]{fontenc} 
\usepackage{siunitx}
\usepackage{array, booktabs}
\usepackage{array, booktabs}
\usepackage{multirow}
\usepackage[T1]{fontenc}
\usepackage{babel}
\usepackage{float}
\usepackage{romannum}
\usepackage{authblk}
\usepackage{mathtools}

\newcommand{\Z}{\mbox{$\mathbb Z$}}	
\newcommand{\Q}{\mbox{$\mathbb Q$}}	

\usepackage{blindtext}
\usepackage{sectsty}
\sectionfont{\centering}

\usepackage{vmargin}
\setmarginsrb     { 0.9in}  
{ 0.6in}  
{ 0.9in}  
{ 0.6in}  
{  20pt}  
{0.25in}  
{   9pt}  
{ 0.3in}  
\newtheorem{theorem}{Theorem}[section]
\newtheorem{lemma}[theorem]{Lemma}
\newtheorem{corollary}[theorem]{Corollary}
\theoremstyle{definition}

\newtheorem{example}[theorem]{Example}

\theoremstyle{remark}
\newtheorem{remark}[theorem]{Remark}

\begin{document}
	\addcontentsline{toc}{chapter}{\bf Notation}
	\addcontentsline{toc}{chapter}{\bf Asymptotic Notation for Runtime Analysis}
	\pagenumbering{arabic}
	\author[]{{\small { ANKITA JINDAL\footnote{Stat-Math Unit, ISI
					Bangalore Centre, Bangalore, India. Email : ankitajindal1203@gmail.com}  }~~AND~ SUDESH KAUR KHANDUJA\footnote{Corresponding author} \footnote{Indian Institute of Science Education and Research Mohali, SAS Nagar, India  \textsc{\&} Department of Mathematics, Panjab University, Chandigarh, India. Email : skhanduja@iisermohali.ac.in}} \\{\small{(Dedicated to Professor Andr\'{e} Leroy on his $65^{th}$ birthday)} }} 
	
\date{}
\renewcommand\Authands{}

\title{\large{\bf{\textsc{ An extension of Schur's irreducibility result }}}}

\maketitle
\begin{center}
	{\large{\bf {\textsc{Abstract}}}}
\end{center}
\noindent 
Let $n\geq 2$ be an integer. Let $\phi(x)$ belonging to $\Z[x]$ be a monic polynomial which is irreducible modulo all primes less than or equal to $n$. Let $a_0(x), a_1(x), \dots, a_{n-1}(x)$ belonging to $\Z[x]$ be polynomials each having degree less than $\deg \phi(x)$ and $a_n$ be an integer. Assume that $a_n$ and the content of $a_0(x)$ are coprime with $n!$. In the present paper, we prove that the polynomial $\sum\limits_{i=0}^{n-1} a_i(x)\frac{\phi(x)^i}{i!}+a_n\frac{\phi(x)^n}{n!}$ is irreducible over the field $\Q$ of rational numbers. This generalizes a well known result of Schur which states that the polynomial $\sum\limits_{i=0}^{n} a_i\frac{x^i}{i!}$ is irreducible over $\Q$ for all $n\geq 1$ when each $a_i\in \Z$ and $|a_0|=|a_n|=1$. The present paper also extends a result of Filaseta thereby leading to a generalization of the classical Sch\"{o}nemann Irreducibility Criterion.
\bigskip


\newpage

\section{Introduction}
A long established theorem of Schur \cite{schur1929a} states that the polynomial $\sum\limits_{i=0}^{n} a_i\frac{x^i}{i!}$ is irreducible over the field $\Q$ of rational numbers for all $n\geq 1$ when each $a_i\in \Z$ and $|a_0|=|a_n|=1$. An alternate proof of this result when $|a_i|=1$ for $i\in\{0,1,\dots,n\}$ was given by Coleman using Newton polygons  with respect to primes (cf. \cite{coleman1987}). In the present paper, we extend the result of Schur as well as of Coleman in different directions using $\phi$-Newton polygons defined in this section and prove the following.

\begin{theorem}\label{genschur}
	Let $n\geq2$ be an integer. Let $\phi(x)$ belonging to $\Z[x]$ be a monic polynomial which is irreducible modulo all primes less than or equal to $n$. Let $a_0(x), a_1(x), \dots, a_{n-1}(x)$ belonging to $\Z[x]$ be polynomials each having degree less than $\deg \phi(x)$ and  $a_n$ be an integer. Assume that $a_n$ and the content{\footnote{The content of a polynomial with coefficients in the ring $\Z$ of integers is the greatest common divisor of its coefficients.}} of $a_0(x)$ are coprime with $n!$. Then the polynomial
	$$\sum\limits_{i=0}^{n-1}a_i(x)\frac{\phi(x)^i}{i!}+a_n\frac{\phi(x)^n}{n!}$$
	is irreducible over $\Q$.
\end{theorem}

It may be pointed out that in the above theorem the assumptions ``$a_n$ and the content of $a_0(x)$ are coprime with $n!$" cannot be dispensed with. For example, consider the polynomial $\phi(x)=x^2+x+1$ which is irreducible modulo 2. The polynomial $F(x)=\frac{\phi(x)^2}{2!}+\phi(x)-4=\frac{1}{2}(\phi(x)+4)(\phi(x)-2)$ is reducible over $\Q.$ Similarly the polynomial $G(x)=12\frac{\phi(x)^2}{2!}-\phi(x)-1=(3\phi(x)+1)(2\phi(x)-1)$ is reducible over $\Q.$\vspace{0.2cm}

We also give below an example to show that Theorem \ref{genschur} may not hold if $a_n$ is replaced by a (monic) polynomial $a_n(x)$ with integer coefficients having degree less than $\deg \phi (x)$. Consider $\phi(x)=x^2+x+1$, $a_2(x)=x+1$,  $a_1(x)=x+2$ and  $a_0(x)=4x+3$. Then $$a_2(x)\frac{\phi(x)^2}{2!}+a_1(x)\phi(x)+a_0(x)$$
has $-1$ as a root.

Assuming a weaker condition on $\phi(x)$ in Theorem \ref{genschur}, we also prove the following result using the Factorisation Theorem of Ore (cf. \cite[Theorem 1.1]{Kh-Ku} and \cite{ore}).
\begin{theorem}\label{gencoleman}
	Let $n\geq 2$ be an integer. Let $\phi(x)$ belonging to $ \Z[x]$ be a monic polynomial which is irreducible modulo all primes dividing $n$. Let $a_0(x), a_1(x), \dots, a_{n-1}(x)$ belonging to $\Z[x]$ be polynomials each having degree less than $\deg \phi$. Assume that the  content of the polynomial $\prod\limits_{i=0}^{n-1}a_i(x)$ is coprime with $n$. Then the polynomial
	$$\sum\limits_{i=0}^{n-1}a_i(x)\frac{\phi(x)^i}{i!}+\frac{\phi(x)^n}{n!}$$
	is irreducible over $\Q$.
\end{theorem}

We wish to point out that analogues of Theorems \ref{genschur} and \ref{gencoleman} do not hold for $n=1$ because if $\phi(x)\in \Z[x]$ is a monic polynomial of degree $m \geq 2$, then the polynomial $\phi(x)-(\phi(x)-x^m)$ is reducible over $\Q$.


In the proof of Theorem \ref{genschur}, a key role is played by Theorem \ref{genfilaseta} which extends a well known result of Filaseta \cite[Lemma 2]{filaseta1995}. For stating this theorem, we first introduce the notion of $\phi$-Newton polygon with respect to a prime number modulo which $\phi(x)\in \Z[x]$ is irreducible.

For a prime $p$, $v_p$ will denote the $p$-adic valuation of $\Q$ defined for any non-zero integer $c$ to be the highest power of $p$ dividing $c$. We shall denote by $v_p^x$ the Gaussian valuation extending $v_p$ defined on the polynomial ring $\Z[x]$ by 
\begin{align*}
v_p^x(\sum\limits_{i}c_ix^i)=\min_{i}\{v_p(c_i)\}, \ c_i\in \Z.
\end{align*}
If $\phi(x)$ is a fixed monic polynomial with coefficients in $\Z$, then any $f(x)\in \Z[x]$ can be uniquely written as a finite sum $\sum\limits_{i}f_i(x)\phi(x)^i$ with $\deg f_i(x)< \deg \phi(x)$ for each $i$; this expansion will be referred to as the $\phi$-expansion of $f(x)$.

Let $\phi(x)\in \Z[x]$ be a monic polynomial which is irreducible modulo a given prime $p$. Let $f(x)$ belonging to $\Z[x]$ be a polynomial having $\phi$-expansion $\sum\limits_{i=0}^{n}f_i(x) \phi(x)^i$ with $f_0(x)f_n(x) \neq 0$. Let $P_i$ stand for the point in the plane having coordinates $(i, v_p^x(f_{n-i}(x)))$ when $f_{n-i}(x)\ne 0$, $0\leq i \leq n$. Let $\mu_{ij}$ denote the slope of the line joining the points $P_i$ and $P_j$ if $f_{n-i}(x)f_{n-j}(x)\ne 0$. Let $i_1$ be the largest index $0< i_1 \leq n$ such that 
\begin{align*}
	\mu_{0i_1}=\min \{\mu_{0j}\ |\ 0<j \leq n,\ f_{n-j}(x)\ne 0\}.
\end{align*}
If $i_1<n$, let $i_2$ be the largest index $i_1< i_2 \leq n$ satisfying 
\begin{align*}
	\mu_{i_1i_2}=\min \{\mu_{i_1j}\ |\ i_1<j \leq n,\ f_{n-j}(x)\ne 0\}
\end{align*}
and so on. The $\phi$-Newton polygon of $f(x)$ with respect to $p$ is the polygonal path having segments $P_0P_{i_1}, P_{i_1}P_{i_2}, \dots, P_{i_{k-1}}P_{i_k}$ with $i_k=n$. These segments are called the edges of the $\phi$-Newton polygon of $f(x)$ and their slopes from left to right form a strictly increasing sequence. The $\phi$-Newton polygon with respect to $p$ minus the horizontal part (if any) is called its principal part.

With the above notation, we prove the following theorem, which was proved by Filaseta \cite[Lemma 2]{filaseta1995} in 1995 in the particular case when $\phi(x)=x$; the proof given here is similar to the one given by him.

\begin{theorem} \label{genfilaseta}
	Let $n, k$ and $\ell$ be integers with $0\leq \ell<k\leq \frac n2$ and $p$ be a prime. Let $\phi(x)\in \Z[x]$ be a monic polynomial which is irreducible modulo $p$. Let $f(x)$ belonging to $\Z[x]$ be a monic polynomial not divisible by $\phi(x)$ having $\phi$-expansion $\sum\limits_{i=0}^{n}f_i(x) \phi(x)^i$ with $f_n(x)\ne 0$. Assume that $v_p^x(f_i(x))> 0$ for $0\leq i\leq n-\ell-1$ and the right-most edge of the $\phi$-Newton polygon of $f(x)$ with respect to $p$ has slope less than $\frac 1k$. Let $a_0(x), a_1(x), \dots,a_n(x)$ be polynomials over $\Z$ satisfying the following conditions. 
	\begin{itemize}
		\item[(i)] $\deg a_i(x)< \deg \phi(x)-\deg f_i(x)$ for $0\leq i \leq n$, 
		\item[(ii)] $v_p^x(a_0(x))=0$, i.e., the content of $a_0(x)$ is not divisible by $p$,
		\item[(iii)] the leading coefficient of $a_n(x)$ is not divisible by $p$.
	\end{itemize}
	Then the polynomial $\sum\limits_{i=0}^{n} a_i(x) f_i(x) \phi(x)^i$	does not have a factor in $\Z[x]$ with degree lying in the interval $[(\ell+1)\deg \phi, (k+1)\deg \phi).$
\end{theorem}

The classical Sch\"{o}nemann Irreducibility Criterion (cf. \cite{schonemann}, \cite[Chapter 3, Theorem D]{riben}) stated below will be quickly deduced from the above theorem.
\begin{corollary}[Classical Sch\"{o}nemann Irreducibility Criterion]\label{corollary}
Let $\phi(x)\in\Z[x]$ be a monic polynomial which is irreducible modulo a given prime $p$. Let $f(x)$ belonging to $\Z[x]$ be of the form $f(x)= \phi(x)^n+pM(x)$ where $M(x)\in\Z[x]$ has degree less than $n \deg \phi$. If $\phi(x)$ is coprime with $M(x)$ modulo $p$, then $f(x)$ is irreducible over $\Q$.
\end{corollary}

\section{Proof of Theorem \ref{gencoleman}}
The following simple result is well known (see \cite[p. 122]{burton}, \cite{coleman1987}). Its proof is omitted.
\begin{lemma}\label{legendre}
	Let $p$ be a prime and $m$ be a positive integer. If $m=a_0+a_1p+\cdots+a_rp^r$ where $\ 0\leq a_i <p$ for each $i$, then
	\begin{align*}
		v_p(m!)=\frac{m-(a_0+a_1+\cdots+a_r)}{p-1}.
	\end{align*}
\end{lemma}

\begin{proof}[Proof of Theorem \ref{gencoleman}]
Set $a_n(x)=1$ and consider the monic polynomial $F(x)\in \Z[x]$ defined by $$F(x)=\sum\limits_{i=0}^{n}\frac{n!}{i!}a_i(x)\phi(x)^i.$$ We shall prove that $F(x)$ is irreducible over $\Q$. Let $p$ be a prime dividing $n$. By hypothesis, the content of $a_i(x)$ is coprime with $p$ for $0\leq i\leq n-1$, i.e., $v_p^x(a_i(x))=0$ for $0\leq i\leq n-1$. Recall that $a_n(x)=1$. Therefore the $\phi$-Newton polygon of $F(x)$ with respect to $p$ is the polygonal path formed by the lower edges along the convex hull of points of the set $S$ defined by
	$$S=\{(i,v_p(n!/(n-i)!))\ |\ 0\leq i \leq n\}.$$
	Write $$n=c_1p^{m_1}+c_2p^{m_2}+\cdots +c_s p^{m_s},$$ where $0<m_1<m_2<\cdots <m_s$, $0<c_i<p$ for each $i$. Set $z_0=0$ and $$z_i=c_1p^{m_1}+\cdots +c_ip^{m_i}$$ 
	for $1\leq i \leq s$.
	As in \cite{coleman1987}, making use of Lemma \ref{legendre}, it can be easily shown that the polygonal path along the lower convex hull of points of the set $S$ consists of $s$ edges. The $i^{th}$ edge from left to right is the segment joining the points $$(z_{i-1}, v_p(n!/(n-z_{i-1})!)),\ (z_i, v_p(n!/(n-z_i)!)).$$
	Again using Lemma \ref{legendre}, we see that the slope $\lambda_i$ of the $i^{th}$ edge of the $\phi$-Newton polygon of $F(x)$ is given by
	\begin{align*}
		\lambda_i
		&=\frac{-v_p((n-z_{i})!)+v_p((n-z_{i-1})!)}{z_i-z_{i-1}}\\
		&=\frac{z_{i}+(c_{i+1}+\cdots+c_{s})-z_{i-1}-(c_i+\cdots+c_{s})}{(z_i-z_{i-1})(p-1)}.
	\end{align*}
So
\begin{align}\label{2e1}
	\lambda_i=\frac{c_{i}p^{m_{i}}-c_{i}}{c_{i}p^{m_{i}}(p-1)}=\frac{p^{m_{i}}-1}{p^{m_{i}}(p-1)}.
\end{align}
Since $0<m_1<\cdots <m_s$, it is immediate from \eqref{2e1} that $p^{m_1}$ divides the denominator of slope of each edge of the $\phi$-Newton polygon of $F(x)$ with respect to $p$. Therefore in view of Ore's Factorisation Theorem \cite[Theorem 1.1(ii)]{Kh-Ku}, $p^{m_1} \deg \phi$ divides the degree of each factor of $F(x)$ over the field $\Q_p$ of $p$-adic numbers. Keeping in mind that $p^{m_1}$ is the exact power of an arbitrary prime $p$ dividing $n$, the assertion proved above implies that $n \deg \phi$ divides the degree of each irreducible factor of $F(x)$ over $\Q$. Hence $F(x)$ having degree $n \deg \phi$ is irreducible over $\Q$.
\end{proof}

\begin{example}
Consider $\phi(x)=x^3-x^2+1$ or $x^4-x-1$. It can be easily checked that in both situations $\phi(x)$ is irreducible modulo 2 as well as modulo 3. Let $a_0(x), a_1(x), \dots, a_5(x)$ belonging to $\Z[x]$ be polynomials with degree less than $\deg \phi$ and each having content coprime with $6$. By Theorem \ref{gencoleman}, the polynomial 
\begin{align*}
	F(x)=\phi(x)^6+\frac{6!}{5!} a_5(x)\phi(x)^5+\frac{6!}{4!} a_4(x)\phi(x)^4+ \frac{6!}{3!} &a_3(x)\phi(x)^3+ \frac{6!}{2!} a_2(x)\phi(x)^2\\
	&+ 6!  a_1(x) \phi(x) + 6!  a_0(x)
\end{align*}
is irreducible over $\Q$.
\end{example}

\section{Proof of Theorem \ref{genfilaseta}}
	
We shall use the following lemma whose first assertion is proved in \cite[Lemma 2.4]{Kh-Ku}.\\

\noindent\textbf{Lemma 3.A. }{\it Let $\phi(x)$ belonging to $\Z[x]$ be a monic polynomial which is irreducible modulo a given prime $p$. Let $g(x), h(x)$ belonging to $\Z[x]$ be polynomials not divisible by $\phi(x)$ having leading coefficients coprime with $p$. Then the following hold.
\begin{itemize}
	\item [(i)] The principal part of the $\phi$-Newton polygon of $g(x)h(x)$ with respect to $p$ is obtained by constructing a polygonal path beginning with a point on the non-negative side of $x$-axis and using translates of edges of the principal part in the $\phi$-Newton polygons of $g(x), h(x)$ in the increasing order of slopes.
	\item [(ii)] The length of the edge of the $\phi$-Newton polygon of $g(x)h(x)$ with respect to $p$ having slope zero is either the sum of the lengths of the edges of the $\phi$-Newton polygons of $g(x)$, $h(x)$ with respect to $p$ having slope zero or it exceeds this sum by one.
\end{itemize}}
\begin{proof} We only prove assertion (ii).
For any $f(x)\in \Z[x]$, let $\bar{f}(x)$ denote the polynomial obtained by replacing each coefficient of $f(x)$ modulo $p$. Let $g(x)=\sum\limits_{i=0}^{t}g_i(x)\phi(x)^i$ and $h(x)=\sum\limits_{i=0}^{u}h_i(x)\phi(x)^i$ be the $\phi$-expansions of $g(x), h(x)$ respectively with $g_t(x)h_u(x)\ne 0$. Let $r\geq 0$, $s\geq 0$ denote respectively the length of the edges having slope zero of the $\phi$-Newton polygons of $g(x), h(x)$ with respect to $p$. Then
\begin{align*}
	v_p^x(g_{t-r}(x))=0, \ v_p^x(g_{t-i}(x))>0 \textrm{ for } i \in \{r+1, \dots, t\}
\end{align*}
and
\begin{align*} 
	v_p^x(h_{u-s}(x))=0, \ v_p^x(h_{u-i}(x))>0 \textrm{ for } i \in \{s+1, \dots, u\}.
\end{align*}
So the highest power of $\bar{\phi}(x)$ dividing $\bar{g}(x)$, $\bar{h}(x)$ is $t-r$ and $u-s$ respectively. Consequently the highest power of $\bar{\phi}(x)$ dividing $\bar{g}(x)\bar{h}(x)$ is $t+u-r-s$. Note that if we write the $\phi$-expansion of $g(x)h(x)$ as $\sum\limits_{i=0}^{k}d_i(x)\phi(x)^i$ with $d_k(x)\ne 0$, $\deg d_i(x) < \deg \phi(x)$, then either $k=t+u$ or $k=t+u+1$. Therefore the length of the edge having slope zero in the $\phi$-Newton polygon of $g(x)h(x)$ is either $r+s$ or $r+s+1$; this proves assertion (ii)  of the lemma.
\end{proof}

\begin{proof}[Proof of Theorem \ref{genfilaseta}]
Write 
\begin{align*}
	F(x)=\sum\limits_{i=0}^{n} a_i(x) f_i(x) \phi(x)^i.
\end{align*}
We first consider the case when $a_i(x)=1$ for all $i \in \{0,1,\dots, n\}$ so that $F(x)=f(x)$. Recall that $f(x) \in \Z[x]$ is monic. Suppose to the contrary that $f(x)$ has a factor over $\Z$ with degree lying in the interval $[(\ell+1)\deg \phi, (k+1)\deg \phi)$. Then there exist monic polynomials $g(x), h(x) \in \Z[x]$ with $f(x)=g(x)h(x)$ and 
\begin{align}\label{3e1}
	(\ell+1)\deg \phi \leq \deg g <(k+1) \deg \phi.
\end{align}
Let $\sum\limits_{i=0}^{r}g_i(x) \phi(x)^i$ be the $\phi$-expansion of $g(x)$ with $g_r(x)\ne 0$. Note that $r$ is the sum of lengths of horizontal projections of all edges of the $\phi$-Newton polygon of $g(x)$ with respect to $p$. We have
\begin{align}\label{3e2}
	r\deg \phi \leq \deg g= \deg g_r+ r\deg \phi \leq (r+1) \deg \phi-1.
\end{align}
The above inequality together with \eqref{3e1} implies that
\begin{align*}
	r+1\geq \frac{\deg g+1}{\deg \phi}\geq \frac{(\ell+1)\deg \phi+1}{\deg \phi}=\ell+1+\frac {1}{\deg \phi},
\end{align*}
which shows that
\begin{align}\label{3e3}
	r\geq \ell+\frac {1}{\deg \phi}>\ell.
\end{align}

We now consider the $\phi$-Newton polygon of $f(x)$ (with respect to $p$). By hypothesis, each edge of this Newton polygon has slope in the interval $[0, \frac 1k)$. For now, we consider an edge of the $\phi$-Newton polygon of $f(x)$ which has positive slope. Let $(a,b)$ and $(c,d)$ be two points with integer entries on such an edge. Then slope of the line joining these points is the slope of that edge, so that
\begin{align*}
	\frac{1}{|c-a|}\leq \frac{|d-b|}{|c-a|}<\frac{1}{k}.
\end{align*}
Hence $|c-a|>k$. Therefore any two such points on an edge with positive slope of the $\phi$-Newton polygon of $f(x)$ have their $x$-coordinates separated by a distance strictly greater than $k$. 

Note that in view of \eqref{3e2} and \eqref{3e1}, we have
\begin{align*}
	r\leq \frac{\deg g}{\deg \phi}<k+1.
\end{align*}
The above inequality implies that $r$, which is the sum of lengths of horizontal projections of all edges of the $\phi$-Newton polygon of $g(x)$, is less than $k+1$. So by what has been proved in the above paragraph, the translates of the edges of the $\phi$-Newton polygon of $g(x)$ cannot be found within those edges of the $\phi$-Newton polygon of $f(x)$ which have positive slope. Therefore by Lemma 3.A, the left-most edge of the $\phi$-Newton polygon of $f(x)$ must have slope zero and its length is greater than or equal to $r$. In view of the hypothesis
\begin{align*}
	v_p^x(f_{n-i}(x))> 0 \textrm{ for } i \in \{\ell+1, \dots, n\},
\end{align*}
we see that the length of the edge of the $\phi$-Newton polygon of $f(x)$ having slope zero is less than or equal to $\ell$. On recalling that $\ell<r$ by virtue of \eqref{3e3}, we obtain a contradiction in view of Lemma 3.A and hence the theorem is proved when all $a_i(x)$ are 1.

Next we consider the general case when
\begin{align*}
	F(x)=\sum\limits_{i=0}^{n} a_i(x) f_i(x) \phi(x)^i,
\end{align*}
where the polynomials $a_i(x)$ satisfy the conditions $(i), (ii)$ and $(iii)$. It is clear that the left-most endpoint and the right-most endpoint of the $\phi$-Newton polygon of $F(x)$ are the same as the left-most endpoint and the right-most endpoint of the $\phi$-Newton polygon of $f(x)$. Recall that
\begin{align*}
		v_p^x(a_i(x)f_i(x))\geq v_p^x(f_i(x))&> 0 \quad \quad\textrm{ for } i \in \{0, 1, \dots,n-\ell-1\},
\end{align*}
and
\begin{align*}
	v_p^x(a_i(x)f_i(x))\geq	v_p^x(f_i(x))&\geq  0 \quad \quad \textrm{ for } i \in \{0, 1, \dots,n\}.
\end{align*}
Therefore all the edges of the $\phi$-Newton polygon of $F(x)$ lie above or on the line containing the right-most edge of the $\phi$-Newton polygon of $f(x)$. Since the right-most endpoints of the $\phi$-Newton polygons of $F(x)$ and $f(x)$ are the same, we deduce that the slope of the right-most edge of the $\phi$-Newton polygon of $F(x)$ is less than or equal to the slope of the right-most edge of the $\phi$-Newton polygon of $f(x)$ and hence the former is also less than $\frac 1 k$. Recall that the leading coefficient of $a_n(x)$ is coprime with $p$ by hypothesis. Thus by appealing to the first part of the proof and applying Lemma 3.A, Theorem \ref{genfilaseta} follows.
\end{proof}

\begin{remark}
	It may be pointed out that the proof of Theorem \ref{genfilaseta} carries over verbatim when the polynomials $f(x)$, $a_0(x)$, $\dots$, $a_n(x)$ have coefficients in the valuation ring $R_v$ of a discrete valuation $v$ having value group $\Z$. Note that Lemma 3.A used in the above proof also remains valid for polynomials with coefficients in a discrete valuation ring $R_v$ having leading coefficient a unit in $R_v$. Indeed assertion $(i)$ of Lemma 3.A is already proved for such polynomials in \cite[Lemma 2.4]{Kh-Ku}, whereas the proof of assertion $(ii)$ given here carries over verbatim in the case of such polynomials.
\end{remark} 

\begin{proof}[Proof of Corollary \ref{corollary}]
	Let $f(x)=\sum\limits_{i=0}^{n} f_i(x) \phi(x)^i$ be the $\phi$-expansion of $f(x)$. In view of the hypothesis, we have $f_n(x)=1$, $v_p^x(f_i(x))>0$ for $0 \leq i \leq n-1$ and $v_p^x(f_0(x))=1$. So the $\phi$-Newton polygon of $f(x)$ with respect to $p$ consists of a single edge joining the points $(0,0)$ and $(n,1)$ which has slope $\frac 1 n$. Choose $\ell=0$ and $k $ to be the largest integer not exceeding $\frac n2$. Applying Theorem \ref{genfilaseta}, we conclude that $f(x)$ does not have a factor in $\Z[x]$ with degree in the interval $[\deg\phi,(k+1) \deg\phi )$.
	
	Suppose to the contrary that $f(x)$ is reducible over $\Q$. Then $f(x)=g(x)h(x)$ where $g(x)$, $h(x)$ are monic polynomials belonging to $\Z[x]$ of positive degrees. Therefore on passing to $\Z/p\Z$, we have $\bar{f}(x)=\bar{g}(x) \bar{h}(x)$. By hypothesis, $\bar{f}(x)=\bar{\phi}(x)^n$. Since $\bar{\phi}(x)$ is irreducible over $\Z/p\Z$, it follows that $\bar{g}(x)=\bar{\phi}(x)^d$ and $\bar{h}(x)=\bar{\phi}(x)^e$ for some positive integers $d$ and $e$. In particular, both $g(x)$ and $h(x)$ have degrees greater than or equal to the degree of $\phi(x)$. Since one of $g(x)$ or $h(x)$ has degree strictly less than $(k+1)\deg \phi$, we obtain a contradiction in view of what has been proved above.
\end{proof}

\section{Proof of Theorem \ref{genschur}}

The following result to be used in the sequel was proved independently by  Sylvester \cite{sylvester} and Schur \cite{schur1929b}. This was reproved by Erd\H{o}s in 1934 (cf. \cite{erdos}).\vspace{0.2cm}

\noindent \textbf{Theorem 4.A.} {\it Let $m$ and $k$ be positive integers with $m\geq k$. Then there is a prime $p\geq k+1$ which divides one of the integers $m+1, m+2, \dots, m+k$.}\vspace{0.2cm}

\begin{proof}[Proof of Theorem \ref{genschur}]
	Consider the polynomial $F(x)=\sum\limits_{i=0}^{n-1}\frac{n!}{i!}a_i(x)\phi(x)^i+a_n\phi(x)^n$ with coefficients in $\Z$. In view of the hypothesis, $a_0(x)$ is non-zero. So $F(x)$ is not divisible by $\phi(x)$. Since $\deg a_i(x)<\deg \phi(x)$ for each $i$, it follows that the leading coefficient of $F(x)$ is $a_n$, which is coprime with $n!$ by hypothesis. Thus the content $c$ (say) of $F(x)$ is coprime with $n!$. It suffices to prove that $\frac{F(x)}{c}$ is irreducible over $\Z$. We divide the proof of the theorem in two steps.\vspace{0.3cm}
	
	\noindent \textbf{Step I.} In this step, we show that $\frac{F(x)}{c}$ does not have a non-constant factor over $\Z$ with degree less than $\deg \phi(x)$. Suppose to the contrary that $\frac{F(x)}{c}=g(x)h(x)$ where $g(x), h(x)$ are non-constant polynomials having coefficients in $\Z$ with $\deg g(x)< \deg \phi(x)$. Fix a prime $q$ dividing $n$. Then as pointed out above, the leading coefficient of $F(x)$ and hence those of $g(x)$ and $h(x)$ are coprime with $q$. Therefore on passing to $\Z/q\Z$, we see that the degree of $\bar{g}(x)$ is same as that of $g(x)$. Thus $\deg \bar{g}(x)$ is positive and less than $\deg \phi(x)$. This is impossible because $\bar{g}(x)$ is a divisor of $\frac{\bar{F}(x)}{\bar{c}}=\frac{\bar{a_n}}{\bar{c}}\bar{\phi}(x)^n$ and $\bar{\phi}(x)$ is irreducible over $\Z/q\Z$.
	This contradiction proves the assertion of Step 1.\vspace{0.3cm}

	\noindent \textbf{Step II.} In this step, we show that $F(x)$ does not have a factor in $\Z[x]$ with degree lying in the interval $[k \deg \phi, (k+1) \deg \phi)$ for any integer $k$, $1 \leq k \leq \frac n2$. This along with Step I will prove the irreducibility of $F(x)$ over $\Q$.

	Suppose to the contrary that $F(x)$ has a factor in $\Z[x]$ with degree lying in the interval $[k \deg \phi, (k+1) \deg \phi)$ for some positive integer $k\leq \frac n2$. By Theorem 4.A, there exists a prime $p\geq k+1$ dividing one of the members of the set $$\{n-k+1, \dots, n-1,n\}.$$ Choose $\ell \in \{0, 1, \dots, k-1\}$ such that $p|(n-\ell)$; such an integer $\ell$ is unique. 
	
	Define a monic polynomial $f(x)\in \Z[x]$ by $f(x)=\sum\limits_{i=0}^{n}c_i\phi(x)^i$ where $c_i=\frac{n!}{i!}$.  Note that $p|c_i$ for $i \in \{0,1, \dots, n-\ell-1\}$. Let $\lambda$ denote the slope of the right-most edge of the $\phi$-Newton polygon of $f(x)$ with respect to $p$, where the prime $p$ is as above. We shall show that
	\begin{align}\label{4e8}
		\lambda<\frac {1}{k}.
	\end{align}
	Note that the prime $p$, being a divisor of $n!$, does not divide the content of $a_0(x)$ nor does it divide $a_n$ by virtue of the hypothesis. Therefore \eqref{4e8} together with Theorem \ref{genfilaseta} immediately yields that $F(x)$ does not have a factor in $\Z[x]$ with degree lying in the interval $[(\ell+1) \deg \phi, (k+1) \deg \phi)$. Since $\ell+1\leq k$, this will imply that $F(x)$ does not have a factor in $\Z[x]$ with degree in the interval $[k \deg \phi, (k+1) \deg \phi)$, which leads to a contradiction. This contradiction will complete the proof of the theorem.
	 
	It only remains to verify \eqref{4e8}. By definition
	\begin{align*}
		\lambda=\max\limits_{1\leq i \leq n}\left\{\frac{v_p(c_0)-v_p(c_i)}{i}\right\}.
	\end{align*}
	Note that $v_p(c_0)-v_p(c_i) = v_p(n!)-v_p(n!/i!)=v_p(i!)$. It is immediate from Lemma \ref{legendre} that
	\begin{align*}	
		\frac{v_p(i!)}{i} <\frac{1}{p-1};
	\end{align*}
	consequently $\lambda<\frac {1}{p-1}$ which in view of the fact $p\geq k+1$ proves \eqref{4e8}.
\end{proof}

\begin{example}
	Consider $\phi(x)=x^3-x^2+1$ or $x^4-x-1$. It can be easily checked that $\phi(x)$ is irreducible modulo 2 as well as modulo 3. Let $a_0(x)$, $a_1(x)$, $a_2(x)$, $a_3(x)$ belonging to $\Z[x]$ be polynomials each having degree less than $\deg \phi$ and $a_4$ be an integer. Assume that $a_4$ and the content of $a_0(x)$ are coprime with $6$. By Theorem \ref{genschur} the polynomial 
	\begin{align*}
		F(x)= a_4\phi(x)^4+ \frac{4!}{3!}a_3(x)\phi(x)^3+ \frac{4!}{2!}a_2(x)\phi(x)^2 +4! a_1(x) \phi(x) + 4! a_0(x)
	\end{align*}
	is irreducible over $\Q$.
\end{example}

\begin{example}
	Consider $\phi(x)=x^4-x-1$. It can be easily checked that $\phi(x)$ is irreducible modulo 2, 3 and 5. Let $a_i(x) \in \Z[x]$ be polynomials each having degree less than 4 for  $0\leq i\leq 5$ and $a_6$ be an integer. Assume that $a_6$ and the content of $a_0(x)$ are coprime with $30$. By Theorem \ref{genschur}, the polynomial 
	\begin{align*}
		F(x)=a_6\phi(x)^6+ \frac{6!}{5!}a_5(x)\phi(x)^5+ \frac{6!}{4!}a_4(x)\phi(x)^4+ &\frac{6!}{3!}a_3(x) \phi(x)^3+ \frac{6!}{2!} a_2(x)\phi(x)^2\\
		&\hspace{0.5cm} +6! a_1(x) \phi(x) + 6! a_0(x)
	\end{align*}
	is irreducible over $\Q$.
\end{example}

It may be pointed out that irreducibility of the polynomials $F(x)$ in Examples 2.2 and 4.2 
does not seem to follow from any known irreducibility criterion  (cf. \cite{brown2008}, \cite{JakBLMS}, \cite{JakAMS}, \cite{JakJA}, \cite{JakAM}, \cite{Ja-Sa} and \cite{Jho-Kh}) including Generalized Sch\"{o}nemann Irreducibility Criterion{\footnote{This criterion for polynomials over $\Z$ asserts that if $\phi(x)$ belonging to $\Z[x]$ is a monic polynomial which is irreducible modulo a fixed prime $p$ and if the $\phi$-expansion of a polynomial $f(x)$ belonging to $\Z[x]$ given by $\sum\limits_{i=0}^{n} f_i(x)\phi(x)^i$ satisfies (i) $f_0(x)\ne 0$, $f_n(x)=1$, (ii) $\frac{v_p^x(f_i(x))}{n-i}\geq \frac{v_p^x(f_0(x))}{n}>0$ for $0\leq i \leq n-1$, (iii) $v_p^x(f_0(x)), n$ are coprime, then $f(x)$ is irreducible over $\Q$.
}}. \\

{\bf Acknowledgements}\\  The first author is grateful to the National Board for Higher Mathematics, Department of Atomic Energy, India for postdoctoral fellowship. This work was mainly done during the visit of the first author to IISER Mohali in February 2023. She is thankful to IISER Mohali for the facilities provided to her.  The second author is thankful to the Indian National Science Academy for Honorary Scientistship.

\end{document}